\documentclass[11pt,a4paper]{article}
\usepackage{amsfonts}
\usepackage{epsf,epsfig,amsfonts,amsgen,amsmath,amstext,amsbsy,amsopn,amsthm}
\usepackage{ebezier,eepic}
\usepackage{color}
\newtheorem{theorem}{Theorem}

\newtheorem{lemma}{Lemma}
\newtheorem{false statement}{False statement}

\theoremstyle{definition}

\newtheorem{claim}{Claim}

\newtheorem{case}{Case}
\newtheorem{subcase}{Case}[case]

\newcommand{\cl}{{\rm cl}}

\begin{document}

\title{\bf\Large Degree conditions restricted to induced paths for hamiltonicity of claw-heavy graphs
\thanks{Supported by NSFC (11271300) and the project NEXLIZ --
CZ.1.07/2.3.00/30.0038, which is co-financed by the European Social
Fund and the state budget of the Czech Republic. Email addresses:
{\tt libinlong@mail.nwpu.edu.cn} (B. Li), {\tt ningbo-maths@163.com}
(B. Ning), {\tt sgzhang@nwpu.edu.cn} (S. Zhang).}}

\date{}

\author{Binlong Li$^{a,c}$, Bo Ning$^{b}$ and Shenggui Zhang$^{a}$\thanks{Corresponding author.}\\[2mm]
\small $^a$ Department of Applied Mathematics, School of Science, \\
\small Northwestern Polytechnical University, Xi'an, Shaanxi 710072,
P.R.~China\\
\small $^b$ Center for Applied Mathematics, \\
\small Tianjin University, Tianjin 300072, P. R.~China\\
\small $^c$ European Centre of Excellence NTIS, \\
\small University of West Bohemia, 30614 Pilsen, Czech Republic}
\date{}

\maketitle

\begin{abstract}
Broersma and Veldman proved that every 2-connected claw-free and $P_6$-free graph is hamiltonian. Chen et al. extended {this result} by proving every 2-connected claw-heavy and $P_6$-free graph is hamiltonian. On the other hand, Li et al. constructed a class of {2-connected graphs which are claw-heavy and $P_6$-\emph{o}-heavy but} not hamiltonian. In this paper we further give some Ore-type degree conditions restricting to induced $P_6$s of a 2-connected claw-heavy graph that can guarantee the graph to be hamiltonian. This improves some previous related results.

\smallskip
\end{abstract}

\section{Introduction}

Throughout this paper, the graphs considered are undirected, finite
and simple. For terminology and notations not defined here, we
refer the reader to Bondy and Murty \cite{Bondy_Murty}.

Let $G$ be a graph. For a given graph $H$, we say that $G$ is
\emph{$H$-free} if $G$ contains no induced subgraph isomorphic to
$H$. In this case, we call $H$ a \emph{forbidden subgraph} of $G$.
Note that if $H_1$ is an induced subgraph of $H_2$, then an
$H_1$-free graph is also $H_2$-free.

The bipartite graph $K_{1,3}$ is {called} the \emph{claw}. Instead of
$K_{1,3}$-free, we say that a graph is \emph{claw-free} if it does
not contain a copy of $K_{1,3}$ as an induced subgraph. As usual, we
use $P_i$ to denote the path of order $i$. Some other special graphs {used in this paper}
are shown in Figure 1.

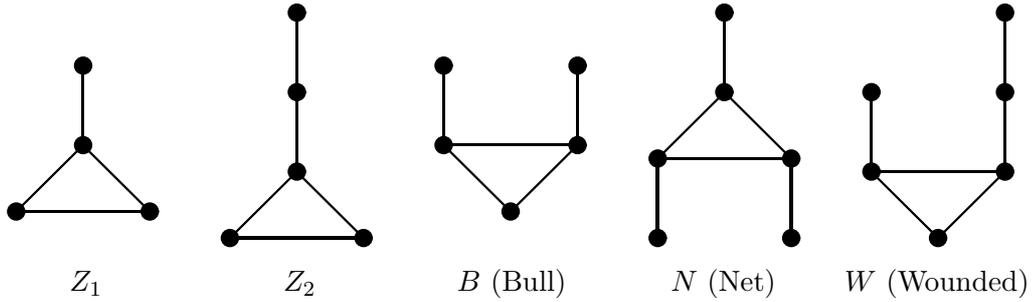
\begin{figure}

\begin{center}
\begin{picture}(410,125)
\thicklines

\put(0,0){\put(20,40){\circle*{6}} \put(70,40){\circle*{6}}
\multiput(45,65)(0,30){2}{\put(0,0){\circle*{6}}}
\put(20,40){\line(1,0){50}} \put(20,40){\line(1,1){25}}
\put(70,40){\line(-1,1){25}} \put(45,65){\line(0,1){30}}
\put(40,10){$Z_1$}}

\put(80,0){\put(20,30){\circle*{6}} \put(70,30){\circle*{6}}
\multiput(45,55)(0,30){3}{\put(0,0){\circle*{6}}}
\put(20,30){\line(1,0){50}} \put(20,30){\line(1,1){25}}
\put(70,30){\line(-1,1){25}} \put(45,55){\line(0,1){60}}
\put(40,10){$Z_2$}}

\put(160,0){\put(45,40){\circle*{6}} \put(45,40){\line(-1,1){25}}
\put(45,40){\line(1,1){25}} \put(20,65){\line(1,0){50}}
\multiput(20,65)(50,0){2}{\multiput(0,0)(0,30){2}{\put(0,0){\circle*{6}}}
\put(0,0){\line(0,1){30}}} \put(25,10){$B$ (Bull)}}

\put(240,0){\multiput(20,30)(50,0){2}{\multiput(0,0)(0,30){2}{\put(0,0){\circle*{6}}}
\put(0,0){\line(0,1){30}}}
\multiput(45,85)(0,30){2}{\put(0,0){\circle*{6}}}
\put(45,85){\line(0,1){30}} \put(20,60){\line(1,0){50}}
\put(20,60){\line(1,1){25}} \put(70,60){\line(-1,1){25}}
\put(25,10){$N$ (Net)}}

\put(320,0){\put(45,30){\circle*{6}} \put(20,55){\line(1,0){50}}
\put(45,30){\line(1,1){25}} \put(45,30){\line(-1,1){25}}
\multiput(20,55)(0,30){2}{\put(0,0){\circle*{6}}}
\multiput(70,55)(0,30){3}{\put(0,0){\circle*{6}}}
\put(20,55){\line(0,1){30}} \put(70,55){\line(0,1){60}}
\put(10,10){$W$ (Wounded)}}

\end{picture}

\caption{Graphs $Z_1,Z_2,B,N$ and $W$}\label{FiZ1Z2BNW}
\end{center}
\end{figure}

Forbidden subgraph conditions for hamiltonicity have been studied
since the early 1980s, but till 1991, Bedrossian \cite{Bedrossian} firstly gave a
characterization of all pairs of forbidden subgraphs for hamiltonian
properties of graphs. First we note that a connected
$P_3$-free graph is complete, and clearly is hamiltonian if it has
at least three vertices. In fact, it is not difficult to see that
$P_3$ is the only connected graph $H$ such that every 2-connected
$H$-free graph is hamiltonian. So the following result of Bedrossian
deals with pairs of forbidden subgraphs, excluding $P_3$.

\begin{theorem} [Bedrossian \cite{Bedrossian}]\label{ThBe}
Let $R,S$ be connected graphs of order at least 3 with $R,S\neq P_3$
and let $G$ be a 2-connected graph. Then $G$ being $R$-free and
$S$-free implies $G$ is hamiltonian if and only if (up to symmetry)
$R=K_{1,3}$ and $S=P_4,P_5,P_6,C_3,Z_1,Z_2,B,N$ or $W$.
\end{theorem}

The above forbidden subgraph conditions for hamiltonicity are
sometimes referred to as \emph{structural conditions}. There is
another type of conditions with respect to hamiltonian properties of
graphs, {so-called} \emph{numerical conditions}, of which degree
conditions may be the most well-known.

Let $G$ be a graph. For a vertex $v\in V(G)$ and a subgraph $H$ of
$G$, we use $N_H(v)$ to denote the set, and $d_H(v)$ the
number, of neighbors of $v$ in $H$, respectively. We call $d_H(v)$
the \emph{degree} of $v$ in $H$. The \emph{distance} between two vertices $x,y\in
V(H)$ in $H$, denoted by $d_H(x,y)$, is the length of a shortest
path between $x$ and $y$ in $H$. When no confusion occurs, we
will denote $N_G(v)$, $d_G(v)$ and $d_G(x,y)$ by $N(v)$, $d(v)$ and
$d(x,y)$, respectively.

The followings are two {well-known} results concerning the degree
conditions for hamiltonicity of graphs.

\begin{theorem} [Dirac \cite{Dirac}]\label{ThDi}
Let $G$ be a graph on $n\geq 3$ vertices. If $d(v)\geq n/2$ for
every $v\in V(G)$, then $G$ is hamiltonian.
\end{theorem}

\begin{theorem} [Ore \cite{Ore}]\label{ThOr}
Let $G$ be a graph on $n\geq 3$ vertices. If $d(u)+d(v)\geq n$ for
every pair of nonadjacent vertices $u,v\in V(G)$, then $G$ is
hamiltonian.
\end{theorem}

It is natural to relax the forbidden subgraph conditions to ones {in which some of the forbidden
subgraphs above} are allowed, but some degree conditions are imposed on the
subgraphs. Broersma et al.
\cite{Broersma_Ryjacek_Schiermeyer} introduced the class of 1-heavy
(2-heavy) graphs by restricting Dirac's condition to induced claws
of a graph. Instead of Broersma et al.'s restriction, \v{C}ada
\cite{Cada} put Ore's condition to induced claws of a graph, and
call it an \emph{o}-heavy graph. {(In this paper, we will call it a
claw-\emph{o}-heavy graph for convenience.)} Li et al.
\cite{Li_Ryjacek_Wang_Zhang} extended \v{C}ada's concept of
claw-\emph{o}-heavy graphs to a more general one.

Let $G$ be a graph on $n$ vertices. Following
\cite{Li_Ryjacek_Wang_Zhang}, for a given graph $H$, $G$ is {called}
$H$-\emph{o}-\emph{heavy} (the authors used the notation `$H$-heavy' in
\cite{Li_Ryjacek_Wang_Zhang}), if every induced copy of $H$ in $G$
has two nonadjacent vertices with degree sum in $G$ at least $n$.
Note that an $H$-free graph is trivially $H$-\emph{o}-heavy, and if
$H_1$ is an induced subgraph of $H_2$, then an $H_1$-\emph{o}-heavy
graph is also $H_2$-\emph{o}-heavy. Following \cite{Ning_Zhang}, we
say that a graph $G$ is $H$-\emph{f}-\emph{heavy} if for every induced copy
$G'$ of $H$ in $G$, and every two vertices $u,v\in V(G')$ with
$d_{G'}(u,v)=2$, there holds $\max\{d(u),d(v)\}\geq |V(G)|/2$. Note
that every claw-\emph{f}-heavy graph is also claw-\emph{o}-heavy.

Li et al. \cite{Li_Ryjacek_Wang_Zhang} completely characterized pairs
of Ore-type heavy subgraphs for a 2-connected graph to be
hamiltonian, which extends Theorem \ref{ThBe}. The main result in
\cite{Li_Ryjacek_Wang_Zhang} is given as follows.

\begin{theorem}[Li et al. \cite{Li_Ryjacek_Wang_Zhang}]\label{ThLiRyWaZh}
Let $R,S$ be connected graphs of order at least 3 with $R,S\neq P_3$
and let $G$ be a 2-connected graph. Then $G$ being $R$-o-heavy and
$S$-o-heavy implies $G$ is hamiltonian if and only if (up to
symmetry) $R=K_{1,3}$ and $S=P_4,P_5,C_3,Z_1,Z_2,B,N$ or $W$.
\end{theorem}

It is easy to see that $P_6$ is the only forbidden subgraph $S$
appearing in Theorem \ref{ThBe} but missing here. Li et al.
\cite{Li_Ryjacek_Wang_Zhang} constructed a class of {2-connected  graphs which are claw-\emph{o}-heavy and $P_6$-\emph{o}-heavy but not
hamiltonian.}

In fact, earlier than Bedrossian \cite{Beineke}, Broersma and
Veldman \cite{Broersma_Veldman} proved that every 2-connected
claw-free and $P_6$-free graph is hamiltonian. Chen et al.
\cite{Chen_Zhang_Qiao} furthermore extended Broersma and Veldman's result as follows.

\begin{theorem}[Chen et al. \cite{Chen_Zhang_Qiao}]\label{ThChWeZh}
Every 2-connected claw-o-heavy and $P_6$-free graph is hamiltonian.
\end{theorem}

So one may ask the question: Which degree conditions can be used to
restrict to all induced copies of $P_6$ in a 2-connected
claw-\emph{o}-heavy graph to make it hamiltonian?

A related result is as follows.

\begin{theorem}[Ning and Zhang \cite{Ning_Zhang}]\label{ThNiZh}
Every 2-connected claw-o-heavy and $P_6$-f-heavy graph is
hamiltonian.
\end{theorem}

One may further ask: Can we still put Ore's condition (or Dirac's condition) to induced $P_6$s
of a graph but with some additional restrictions to guarantee that it is hamiltonian?

Our answers are the following two results.
Note that the first theorem weakens the condition of Theorem \ref{ThNiZh}.

\begin{theorem}\label{ThopHeavy}
Let $G$ be a 2-connected claw-o-heavy graph of order at least $n$.
If for every induced copy of $P_6: v_1v_2\cdots v_6$ in $G$,
$d(v_i)+d(v_j)\geq n$ for some $i\in\{1,2,3\}$ and $j\in\{4,5,6\}$,
then $G$ is hamiltonian.
\end{theorem}

\begin{theorem}
Let $G$ be a 2-connected claw-o-heavy graph of order at least $n$.
If for every induced copy of $P_6: v_1v_2\cdots v_6$ in $G$,
$\max\{d(v_1),d(v_6)\}\geq n/2$, then $G$ is hamiltonian.
\end{theorem}

Now we will go further on this direction.
Before giving our main results, we will first introduce some necessary terminology and notations.

Let $\gamma$ be a graph (possibly with loops) with vertex set
$\mathfrak{I}=\{1,2,3,4,5,6\}$ . We say that a graph $G$ is
$P_6$-$\gamma$-\emph{heavy} if, for every induced copy of $P_6:
v_1v_2v_3v_4v_5v_6$ in $G$, there exist $i,j\in\mathfrak{I}$ such
that $ij\in E(\gamma)$ and $d(v_i)+d(v_j)\geq n$, where $n=|V(G)|$.
Note that if $\gamma'$ is a (spanning) subgraph of $\gamma$, then a
$P_6$-$\gamma'$-heavy graph is also $P_6$-$\gamma$-heavy.

For two graphs $\gamma$ and $\gamma'$ on $\mathfrak{I}$ such that
$ij\in E(\gamma)$ if and only if $(7-i)(7-j)\in E(\gamma')$, we say
$\gamma$ is \emph{symmetrical} to $\gamma'$. Note that if $\gamma$ and
$\gamma'$ are symmetrical to each other, then a graph $G$ is
$P_6$-$\gamma$-heavy if and only if $G$ is $P_6$-$\gamma'$-heavy. If
$\gamma$ is symmetrical to itself, then we say $\gamma$ is
\emph{symmetrical}.

Let $\varepsilon$ be the empty graph on $\mathfrak{I}$. Then a graph
$G$ is $P_6$-free if and only if it is $P_6$-$\varepsilon$-heavy.
Let $\sigma$ be the graph on $\mathfrak{I}$ with edge set
$E(\sigma)=\{ij: |j-i|\geq 2, i,j\in \mathfrak{I}\}$. Then a graph
is $P_6$-\emph{o}-heavy means it is $P_6$-$\sigma$-heavy. Let
$\gamma_1$ be the graph on $\mathfrak{I}$ with edge set $\{ij:
i=1,2,3 \mbox{ and } j=4,5,6\}$. Then Theorem \ref{ThopHeavy} states
that every 2-connected claw-\emph{o}-heavy and
$P_6$-$\gamma_1$-heavy graph is hamiltonian.

The goal of this paper is to find all symmetrical graphs $\gamma$ on
$\mathfrak{I}$ such that every 2-connected claw-\emph{o}-heavy {and}
$P_6$-$\gamma$-heavy graph is hamiltonian.

We describe the graphs $\gamma_1,\gamma_2,\gamma_3$ on
$\mathfrak{I}$ by giving their edge sets (also see Figure 2.):
\begin{align*}
    & E(\gamma_1)=\{14,15,16,24,25,26,34,35,36\};\\
    & E(\gamma_2)=\{11,12,14,15,16,25,26,36,56,66\};\\
    & E(\gamma_3)=\{13,14,15,25,26,36,46\}.
\end{align*}

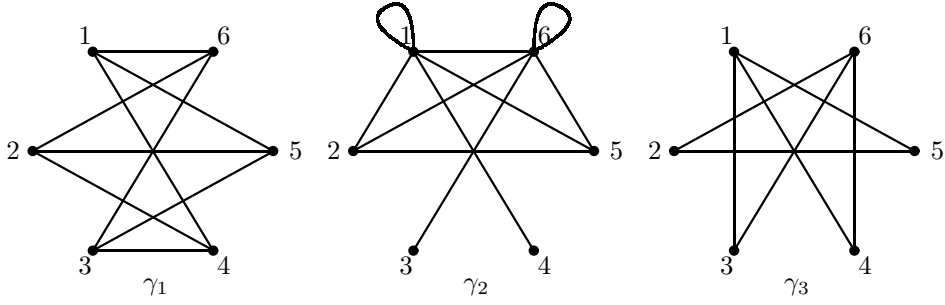
\begin{figure}

\begin{center}
\setlength{\unitlength}{0.75pt}\small
\begin{picture}(480,165)
\thicklines

\put(0,0){\put(80,80){\put(-30,50){\circle*{4}}
\put(-60,0){\circle*{4}} \put(-30,-50){\circle*{4}}
\put(30,-50){\circle*{4}} \put(60,0){\circle*{4}}
\put(30,50){\circle*{4}} \put(-37,54){1} \put(-73,-4){2}
\put(-37,-62){3} \put(32,-62){4} \put(68,-4){5} \put(32,54){6}
\put(-30,50){\line(3,-5){60}} \put(-30,50){\line(9,-5){90}}
\put(-30,50){\line(1,0){60}} \put(-60,0){\line(9,-5){90}}
\put(-60,0){\line(1,0){120}} \put(-60,0){\line(9,5){90}}
\put(-30,-50){\line(1,0){60}} \put(-30,-50){\line(9,5){90}}
\put(-30,-50){\line(3,5){60}} \put(-5,-70){$\gamma_1$}}}

\put(160,0){\put(80,80){\put(-30,50){\circle*{4}}
\put(-60,0){\circle*{4}} \put(-30,-50){\circle*{4}}
\put(30,-50){\circle*{4}} \put(60,0){\circle*{4}}
\put(30,50){\circle*{4}} \put(-37,54){1} \put(-73,-4){2}
\put(-37,-62){3} \put(32,-62){4} \put(68,-4){5} \put(32,54){6}
\put(-30,50){\line(-3,-5){30}} \put(30,50){\line(3,-5){30}}
\put(-30,50){\line(1,0){60}} \put(-60,0){\line(9,5){90}}
\put(-30,50){\line(3,-5){60}} \put(-30,50){\line(9,-5){90}}
\put(-60,0){\line(1,0){120}} \put(30,50){\line(-3,-5){60}}
\qbezier(-30,50)(-30,80)(-43.5,72.5)
\qbezier(-30,50)(-57,65)(-43.5,72.5)
\qbezier(30,50)(30,80)(43.5,72.5) \qbezier(30,50)(57,65)(43.5,72.5)
\put(-5,-70){$\gamma_2$}}}

\put(320,0){\put(80,80){\put(-30,50){\circle*{4}}
\put(-60,0){\circle*{4}} \put(-30,-50){\circle*{4}}
\put(30,-50){\circle*{4}} \put(60,0){\circle*{4}}
\put(30,50){\circle*{4}} \put(-37,54){1} \put(-73,-4){2}
\put(-37,-62){3} \put(32,-62){4} \put(68,-4){5} \put(32,54){6}
\put(-60,0){\line(9,5){90}} \put(-30,50){\line(0,-1){100}}
\put(-30,50){\line(3,-5){60}} \put(-30,50){\line(9,-5){90}}
\put(-60,0){\line(1,0){120}} \put(30,50){\line(-3,-5){60}}
\put(30,50){\line(0,-1){100}} \put(-5,-70){$\gamma_3$}}}

\end{picture}

\caption{Graphs on $\mathfrak{I}$: $\gamma_1$, $\gamma_2$ and
$\gamma_3$}\label{FiGamma}
\end{center}
\end{figure}

The following is our main result. Note that both Theorems 7 and 8
are its corollaries.

\begin{theorem}\label{ThIffi}
Let $\gamma$ be a symmetrical graph on $\mathfrak{I}$. Then every
2-connected claw-o-heavy {and} $P_6$-$\gamma$-heavy graph is hamiltonian
if and only if $\gamma$ is a subgraph of $\gamma_1$,
$\gamma_2$ or $\gamma_3$.
\end{theorem}

\section{Preliminaries}

In this section, we will introduce some preliminaries for the proof
of the `if' part of Theorem \ref{ThIffi}, which are similar to the
ones in \cite{Ning_Zhang_Li}. We first introduce \v{C}ada's closure
theory of claw-\emph{o}-heavy graphs \cite{Cada}, which is an
extension of the closure theory of claw-free graphs invented by
Ryj\'a\v{c}ek \cite{Ryjacek}.

Let $G$ be a graph of order $n$. We say {that} a vertex $x\in V(G)$ is
\emph{heavy} in $G$ if $d(x)\geq n/2$; and a pair of vertices
$\{x,y\}$ is \emph{heavy} in $G$ if $d(x)+d(y)\geq n$. We say {that} a
vertex (a pair of vertices) is \emph{light} if it is not heavy. Note
that if $\{x,y\}$ is a heavy pair, then either $x$ or $y$ is a heavy
vertex.

Let $G$ be a graph and $x\in V(G)$. The \emph{{local completion(?)} of
$G$ at $x$}, denoted by $G'_x$, is the graph obtained from $G$ by
adding all missing {edges in $G[N(x)]$}. Define $B^o_x(G)=\{uv:
\{u,v\}\subset N(x) \mbox{ is a heavy pair of } G\}$. Let $G^o_x$ be
{the} graph with vertex set $V(G^o_x)=V(G)$ and edge set
$E(G^o_x)=E(G)\cup B^o_x(G)$. If $G^o_x[N(x)]$ consists of two
disjoint cliques $C_1$ and $C_2$, then we call a vertex $z\in
V(G)\backslash(\{x\}\cup N(x))$ a \emph{join vertex} of $x$ in $G$
if $\{x,z\}$ is a heavy pair in $G$, and there are two vertices
$y_1\in C_1$ and $y_2\in C_2$ such that $zy_1,zy_2\in E(G)$. The
vertex $x$ is an \emph{o-eligible vertex} of $G$, if $N(x)$ is not a
clique and, $G^o_x[N(x)]$ is connected or, $G^o_x[N(x)]$ consists of
two disjoint cliques and there is some join vertex of $x$.

Let $G$ be a claw-\emph{o}-heavy graph. The \emph{closure} of $G$,
denoted by $\cl(G)$, is the graph such that there is a sequence of
graphs $G_1,G_2,\ldots,G_t$ and a sequence of vertices
$x_1, x_2,\ldots,x_{t-1}$ such that: \\
(1) $G=G_1$, $G_t=\cl(G)$;\\
(2) for $i=1,2,\ldots,t-1$, $G_{i+1}$ is the {local completion(?)} of
$G_i$ at some \emph{o}-eligible vertex $x_i$ of $G_i$; and\\
(3) there is no \emph{o}-eligible vertex in $G_t$.

\begin{theorem}[\v{C}ada \cite{Cada}]\label{ThCa}
Let $G$ be a claw-o-heavy graph. Then\\
(1) the closure $\cl(G)$ is uniquely determined;\\
(2) there is a $C_3$-free graph $H$ such that $\cl(G)$ is the line
graph of $H$; and\\
(3) $G$ is hamiltonian if and only if $\cl(G)$ is.
\end{theorem}

Note that every line graph is claw-free (see \cite{Beineke}). The
above theorem implies that $\cl(G)$ is a claw-free graph.

Now we will give some terminology and notations firstly introduced
in \cite{Ning_Zhang_Li} by the authors. Let $G$ be a
claw-\emph{o}-heavy graph and $C$ be a maximal clique of $\cl(G)$.
We call $G[C]$ a \emph{region} of $G$. For a vertex $v$ of $G$, we
call $v$ an \emph{interior vertex} if it is contained in only one
region, and a \emph{frontier vertex} if it is contained in two
distinct regions. For two vertices $u,v\in V(G)$, we say $u$ and $v$
are \emph{associated} if $u,v$ are contained in a common region of
$G$; otherwise $u$ and $v$ are \emph{dissociated}. We denote by
$I_R$ the set of interior vertices of a region $R$, and by $F_R$ the
set of frontier vertices of $R$.

From \cite{Cada}, it is not difficult to get the following

\begin{lemma}\label{LeClosed}
Let $G$ be a claw-o-heavy graph. Then\\
(1) every vertex is either an interior vertex of a region or
a frontier vertex of two regions;\\
(2) every two regions are either disjoint or have only one common
vertex; and\\
(3) every pair of dissociated vertices have degree sum in $\cl(G)$
(and in $G$) less than $|V(G)|$.
\end{lemma}

We also need the following tools developed in \cite{Ning_Zhang_Li}.

\begin{lemma}\label{LeRegion}
Let $G$ be a claw-o-heavy graph and $R$ be a region of $G$. Then\\
(1) $R$ is nonseparable;\\
(2) if $v$ is a frontier vertex of $R$, then $v$ has a
neighbor in $I_R$ or $I_R=\emptyset$ and $F_R$ is a clique;\\
(3) for any two vertices $u,v\in V(R)$, there is an induced path of
$G$ from $u$ to $v$ such that every internal vertex of the path is
in $I_R$; and\\
(4) for two vertices $u,v$ in $R$, if $uv\notin E(G)$ and $\{u,v\}$
is a heavy pair of $G$, then $u,v$ have two common neighbors in
$I_R$.
\end{lemma}

For two associated vertices $u,v$, by Lemma \ref{LeRegion} (3), we
use $\varPi[u,v]$ to denote a shortest path such that every internal
vertex of it is an interior vertex of the region containing $u,v$.
From Lemma \ref{LeRegion} (4), we can see that every two vertices of
$\varPi[u,v]$ at distance at least 3 in $\varPi[u,v]$ is not a
heavy pair in $G$.

Following \cite{Brousek}, we define $\mathcal{P}$ to be the class of
graphs obtained by taking two vertex-disjoint triangles
$a_1a_2a_3a_1$, $b_1b_2b_3b_1$ and by joining every pair of vertices
$\{a_i,b_i\}$ by a path $P_{k_i}: a_ic_i^1c_i^2\cdots
c_i^{k_i-2}b_i$, for $k_i\geq 3$ or by a triangle $a_ib_ic_i$. We
denote the graphs in $\mathcal{P}$ by $P_{x_i,x_2,x_3}$, where
$x_i=k_i$ if $a_i,b_i$ are joined by a path $P_{k_i}$, and $x_i=T$
if $a_i,b_i$ are joined by a triangle.

The following theorem plays the central role in our proof.

\begin{theorem}[Brousek \cite{Brousek}]\label{ThBr}
Every non-hamiltonian 2-connected claw-free graph contains an
induced subgraph $H\in \mathcal{P}$.
\end{theorem}

\section{Proof of the `if' part of Theorem \ref{ThIffi}}

Let $G$ be a claw-\emph{o}-heavy non-hamiltonian graph of order $n$.
For each $\gamma_k$, $k=1,2,3$, we will show that there exists an
induced $P_6: v_1v_2\cdots v_6$ such that for every edge $ij\in
E(\gamma_k)$, $d(v_i)+d(v_j)<n$. {For convenience, we call such an
induced $P_6$ a \emph{bad $P_6$ to $\gamma_k$} in the following.}

Let $G'=\cl(G)$. By Theorem \ref{ThCa}, $G'$ is claw-free and
non-hamiltonian. By Theorem \ref{ThBr}, let $H\subseteq G'$ be an
induced copy of some graph in $\mathcal{P}$. We denote the vertices
of $H$ as in Section 2. If $x_i=k_i$, then let $a'_i$ be the
neighbor of $a_i$ on $\varPi[a_i,c_i^1]$, $b'_i$ be the neighbor of
$b_i$ on $\varPi[b_i,c_i^{k_i-2}]$, and let
$\varPi_i=\varPi[a_i,c_i^1]c_i^1\varPi[c_i^1,c_i^2]c_i^2\cdots
c_i^{k_i-2}\varPi[c_i^{k_i-2},b_i]$. If $x_i=T$, then let $a'_i$ be
the neighbor of $a_i$ on $\varPi[a_i,c_i]$, $b'_i$ be the neighbor
of $b_i$ on $\varPi[b_i,c_i]$, and let $\varPi_i=\varPi[a_i,b_i]$.
For $1\leq i,j\leq 3$, let $\varPi_{ij}^a=\varPi[a_i,a_j]$ and
$\varPi_{ij}^b=\varPi[b_i,b_j]$. Let $a'_{ij}$ ($b'_{ij}$) be the
neighbor of $a_i$  ($b_i$) on $\varPi_{ij}^a$ ($\varPi_{ij}^b$). Set
$$S=\bigcup_{1\leq i\leq
3}(\{a'_i,b'_i\}\cup\varPi_i)\cup\bigcup_{1\leq i,j\leq
3}(\varPi_{ij}^a\cup\varPi_{ij}^b).$$

For a path $P$ with the origin $x$, we use $P|_i^x$ (or briefly,
$P|_i$) to denote the subpath of $P$ consisting of the first $i$ edges
of $P$. If $P=v_1v_2\cdots v_p$, then we denote
$\overleftarrow{P}=v_pv_{p-1}\ldots v_1$.

\setcounter{claim}{0}
\begin{claim}\label{ClHeavy}
There is a heavy vertex of $G$ in $S\backslash\{a_i,b_i: 1\leq i\leq
3\}$, or there are two heavy vertices in $\{a_i,b_i: 1\leq i\leq
3\}$.
\end{claim}

\begin{proof}
Up to symmetry, suppose that $a_1$ is the vertex with the largest degree
among all vertices in $\{a_i,b_i: 1\leq i\leq 3\}$. If $G$ has no
heavy vertex in $S$ or has the only one heavy vertex $a_1$ in
$S$, then
$P=b'_1b_1\varPi_{12}^bb_2\overleftarrow{\varPi_2}a_2\varPi_{23}^aa_3a'_3$
is an induced path of order at least 6 and each vertex of $P$ is not
heavy in $G$. Thus $P|_5$ is a bad $P_6$ to every $\gamma_k$.
\end{proof}

Note that any two heavy vertices are associated. Up to symmetry, we
have the following cases:

\begin{case}
There is a heavy vertex in
$\{a'_1,b'_1\}\cup(V(\varPi_1)\backslash\{a_1,b_1\})$, or both $a_1$
and $b_1$ are heavy.
\end{case}

Clearly every heavy vertex of $G$ contained in $S$ is in
$\{a'_1,b'_1\}\cup V(\varPi_1)$. Also clearly either $\bigcup_{1\leq
i,j\leq 3}V(\varPi_{ij}^a)$ or $\bigcup_{1\leq i,j\leq
3}V(\varPi_{ij}^b)$ contains no heavy pair of $G$. We suppose
without loss of generality that $\bigcup_{1\leq i,j\leq
3}V(\varPi_{ij}^a)$ contains no heavy pair of $G$. Let
$Q_1=a'_1a_1\varPi_{12}^aa_2\varPi_2b_2\varPi_{23}^bb_3b'_3$. Then
$\overleftarrow{Q_1}|_5$ is a bad $P_6$ to $\gamma_1$.

Suppose now that $\varPi_1=a_1x_1x_2\cdots x_{p-1}b_1$, where $p$ is
the length of $\varPi_1$.

\begin{subcase}
$p=1$, i.e., $\varPi_1=a_1b_1$.
\end{subcase}

Let $Q_2=(\varPi_{12}^aa_2a'_2)|_2$ and
$Q'_2=(\varPi_{13}^bb_3b'_3)|_2$. Then
$\overleftarrow{Q_2}a_1b_1Q'_2$ is a bad $P_6$ to $\gamma_2$,
$\gamma_3$.

\begin{subcase}
$p=2$, i.e., $\varPi_1=a_1x_1b_1$.
\end{subcase}

Let $Q_2=(\varPi_{12}^bb_2b'_2)|_2$. Then
$a'_{13}a_1x_1b_1\overleftarrow{Q_2}$ is a bad $P_6$ to $\gamma_2$,
$\gamma_3$.

\begin{subcase}
$p=3$, i.e., $\varPi_1=a_1x_1x_2b_1$.
\end{subcase}

Note that $\{a_1,b_1\}$ is light. Suppose first that $\bigcup_{1\leq
i,j\leq 3}V(\varPi_{ij}^a)$ contains a heavy pair of $G$. Then $b_1$
is heavy, $a_1$ is light, $\{a_1,x_1\}$ is light and $\{a_1,x_2\}$
is light. Let $Q_2=(b_1\varPi_{12}^bb_2b'_2)|_2$. Then
$a_1x_1x_2b_1\overleftarrow{Q_2}$ is a bad $P_6$ to $\gamma_5$,
$\gamma_2$. Let
$Q_3=x_1a_1\varPi_{12}^aa_2\varPi_2b_2\varPi_{23}^bb_3b'_3$. Then
$Q_3|_5$ is a bad $P_6$ to $\gamma_3$.

Now we suppose that $\bigcup_{1\leq i,j\leq 3}V(\varPi_{ij}^a)$
contains no heavy pairs of $G$. Recall that $\bigcup_{1\leq i,j\leq
3}V(\varPi_{ij}^b)$ contains no heavy pairs of $G$, either
$Q_2=a'_{12}a_1x_1x_2b_1b'_{13}$ is a bad $P_6$ to $\gamma_2$,
$\gamma_3$.

\begin{subcase}
$p\geq 4$.
\end{subcase}

If $x_1$ is light, then
$Q_2=(x_1a_1\varPi_{12}^aa_2\varPi_2b_2\varPi_{23}^bb_3b'_3)|_5$
contains no heavy vertices of $G$, and hence is bad to $\gamma_2$,
$\gamma_3$. So we assume that $x_1$ is heavy, and similarly,
$x_{p-1}$ is heavy. This implies that $p=4$ and $a_1,b_1$ are light.

Note that either $\{a_1,x_1\}$ is light or $\{b_1,x_3\}$ is light.
we assume without loss of generality that $\{a_1,x_1\}$ is light.
Thus $Q_2=a_1x_1x_2x_3b_1b'_{12}$ is bad to $\gamma_2$, $\gamma_3$.

\begin{case}
There is a heavy vertex in $\bigcup_{1\leq i,j\leq
3}(V(\varPi_{ij}^a\backslash\{a_i,a_j\})$, or two of
$\{a_1,a_2,a_3\}$ are heavy.
\end{case}

Clearly every heavy vertex of $G$ is in $\bigcup_{1\leq i,j\leq
3}V(\varPi_{ij}^a)$, and at most one of $\{a'_i,b'_i\}\cup
V(\varPi_i)$ contains heavy pairs of $G$. We assume without loss of
generality that both $\{a'_1,b'_1\}\cup V(\varPi_1)$ and
$\{a'_2,b'_2\}\cup V(\varPi_2)$ contain no heavy pairs of $G$.

Let
$Q_1=b'_2b_2\overleftarrow{\varPi_{12}^b}b_1\overleftarrow{\varPi_1}a_1\varPi_{13}^aa_3a'_3$,
then $Q_1|_5$ is a bad $P_6$ to $\gamma_1$.

Suppose now that $\varPi_{12}^a=a_1x_1x_2\cdots x_{p-1}a_2$, where
$p$ is the length of $\varPi_{12}^a$.

\begin{subcase}
$p=1$, i.e., $\varPi_{12}^a=a_1a_2$.
\end{subcase}

Let $Q_2=(a'_1a_1a_2\varPi_2b_2\varPi_{23}^bb_3b'_3)|_5$. Then $Q_2$
is a bad $P_6$ to $\gamma_2$, $\gamma_3$.

\begin{subcase}
$p=2$, i.e., $\varPi_{12}^a=a_1x_1a_2$.
\end{subcase}

Let $Q_2=(a'_2a_2x_1a_1\varPi_1b_1b'_{13})|_5$. Then $Q_2$ is a bad
$P_6$ to $\gamma_2$, $\gamma_3$.

\begin{subcase}
$p=3$, i.e., $\varPi_{12}^a=a_1x_1x_2a_2$.
\end{subcase}

Let $Q_2=a'_1a_1x_1x_2a_2a'_2$. Then $Q_2$ is a bad $P_6$ to
$\gamma_2$, $\gamma_3$.

\begin{subcase}
$p\geq 4$.
\end{subcase}

Let
$Q_3=(x_1a_1\varPi_1b_1\varPi_{12}^bb_2\overleftarrow{\varPi_2}a_2x_{p-1})|_5$.
Then $Q_3$ is a bad $P_6$ to $\gamma_3$.

If one of $a_1,a_2$ is heavy in $G$, say $a_1$ is heavy, then $x_i$
($i\geq 3$) and $a_2$ are light. Thus
$Q_2=(a'_1a_1\varPi_{12}^a)|_5$ is a bad $P_6$ to $\gamma_2$. So we
assume that $a_1,a_2$ are light.

Recall that for each two vertices with distance at least 3 in
$\varPi_{12}^a$, at least one of them is light. This implies that
there exists an integer $i$, $2\leq i\leq p-2$, such that every vertex in
$V(\varPi_{12}^a)\backslash\{x_{i-1},x_i,x_{i+1}\}$ is light. Note
that either $\{x_{i-2},x_{i-1}\}$ or $\{x_{i+1},x_{i+2}\}$ is light
(we set $x_0=a_1$ and $x_p=a_2$). We assume without loss of
generality that $\{x_{i-2},x_{i-1}\}$ is light. Then
$Q_2=(x_{i-2}x_{i-1}\cdots x_{p-1}a_2a'_2)|_5$ is a bad $P_6$ to
$\gamma_2$.

The proof is complete.

\section{Proof of the `only if' part of Theorem \ref{ThIffi}}

Let $\gamma$ be a symmetrical graph on $\mathfrak{I}$ such that
every 2-connected claw-\emph{o}-heavy {and} $P_6$-$\gamma$-heavy graph is
hamiltonian. We will prove that $\gamma$ is a subgraph of
$\gamma_1$, $\gamma_2$ or $\gamma_3$. Assume not. Then for every
$k=1,2,3$, $E(\gamma)\backslash E(\gamma_k)\neq\emptyset$. Note that
the graphs in Figure \ref{FiCounterexample} are claw-\emph{o}-heavy
and non-hamiltonian. Hence they are not $P_6$-$\gamma$-heavy. Let
$P=u_1u_2\cdots u_6$ and $Q=v_1v_2\cdots v_6$ be two induced copies
of $P_6$ in a graph $G$ of order $n$. We say $P$ and $Q$ are
essentially same if for every $i,j\in[1,6]$, $d(u_i)+d(u_j)\geq n$
if and only if $d(v_i)+d(v_j)\geq n$.

\begin{figure}

\begin{center}
\begin{picture}(360,280)

\newcommand{\tuoyuan}[2]{\qbezier(#1,0)(#1,#2)(0,#2)
\qbezier(0,#2)(-#1,#2)(-#1,0) \qbezier(-#1,0)(-#1,-#2)(0,-#2)
\qbezier(0,-#2)(#1,-#2)(#1,0)}

\thicklines

\put(0,140){\put(30,80){\thinlines\tuoyuan{10}{50}}
\put(24,97){$K_r$} \put(30,60){\circle*{4}} \put(26,52){$w$}
\put(70,30){{\thinlines \qbezier(0,10)(-30,0)(-40,0)
\qbezier(0,10)(-30,100)(-40,100) \qbezier(0,50)(-30,0)(-40,0)
\qbezier(0,50)(-30,100)(-40,100) \qbezier(0,90)(-30,0)(-40,0)
\qbezier(0,90)(-30,100)(-40,100)}}
\put(70,40){\multiput(0,0)(0,40){3}{\put(0,0){\circle*{4}}
\put(40,0){\circle*{4}} \put(80,0){\circle*{4}}
\put(0,0){\line(1,0){80}} \qbezier(0,0)(40,20)(80,0)}
\put(80,0){\line(0,1){80}} \qbezier(80,0)(100,40)(80,80)
\put(74,-8){$x$} \put(36,-8){$x'$} \put(-2,-8){$x''$}
\put(74,32){$y$} \put(36,32){$y'$} \put(-2,32){$y''$}
\put(74,72){$z$} \put(36,72){$z'$} \put(-2,72){$z''$}}
\put(65,10){$G_1$ ($r\geq 7$)}}

\put(180,140){\put(30,64){\thinlines\tuoyuan{10}{34}}
\put(30,96){\thinlines\tuoyuan{10}{34}}
\put(30,30){\multiput(0,10)(0,16){6}{\circle*{4}}
\multiput(0,10)(0,32){3}{\qbezier[4](0,0)(0,8)(0,16)}
\put(0,10){\line(1,0){40}} \put(0,26){\line(5,-2){40}}
\put(0,10){\line(1,1){40}} \put(0,26){\line(5,3){40}}
\put(0,42){\line(5,1){40}} \put(0,58){\line(5,-1){40}}
\put(0,74){\line(5,-3){40}} \put(0,90){\line(1,-1){40}}
\put(0,74){\line(5,2){40}} \put(0,90){\line(1,0){40}}
\put(-10,12){$x_1$} \put(-10,28){$x_r$} \put(-10,44){$y_1$}
\put(-10,60){$y_s$} \put(-10,76){$z_1$} \put(-10,92){$z_t$}
\put(10,30){$K_{r+s}$} \put(10,62){$K_{s+t}$}}
\put(70,40){\multiput(0,0)(0,40){3}{\put(0,0){\circle*{4}}
\put(40,0){\circle*{4}} \put(80,0){\circle*{4}}}
\put(0,0){\line(1,0){80}} \qbezier(0,0)(40,-20)(80,0)
\put(0,80){\line(1,0){80}} \qbezier(0,80)(40,100)(80,80)
{\thinlines\put(40,40){\circle{20}} \put(50,37){$K_q$}
\qbezier(0,40)(30,50)(40,50) \qbezier(0,40)(30,30)(40,30)
\qbezier(80,40)(50,50)(40,50) \qbezier(80,40)(50,30)(40,30)}
\put(80,0){\line(0,1){80}} \qbezier(80,0)(100,40)(80,80)
\put(74,3){$x$} \put(36,3){$x'$} \put(-2,3){$x''$} \put(74,32){$y$}
\put(36,32){$y'$} \put(-2,32){$y''$} \put(74,72){$z$}
\put(36,72){$z'$} \put(-2,72){$z''$}} \put(-5,10){$G_2$ ($q\geq 6$,
$r,s\geq q+6$, $t\geq q+r+5$)}}

\put(90,0){\put(30,80){\thinlines\tuoyuan{10}{50}}
\put(24,97){$K_r$} \put(30,60){\circle*{4}} \put(26,52){$w$}
\put(70,30){\multiput(0,10)(0,16){6}{\circle*{4}}
\multiput(0,10)(0,32){3}{\qbezier[4](0,0)(0,8)(0,16)} {\thinlines
\qbezier(0,10)(-30,0)(-40,0) \qbezier(0,10)(-30,100)(-40,100)
\qbezier(0,26)(-30,0)(-40,0) \qbezier(0,26)(-30,100)(-40,100)
\qbezier(0,42)(-30,0)(-40,0) \qbezier(0,42)(-30,100)(-40,100)
\qbezier(0,58)(-30,0)(-40,0) \qbezier(0,58)(-30,100)(-40,100)
\qbezier(0,74)(-30,0)(-40,0) \qbezier(0,74)(-30,100)(-40,100)
\qbezier(0,90)(-30,0)(-40,0) \qbezier(0,90)(-30,100)(-40,100)}
\put(0,10){\line(1,0){40}} \put(0,26){\line(5,-2){40}}
\put(0,42){\line(5,1){40}} \put(0,58){\line(5,-1){40}}
\put(0,74){\line(5,2){40}} \put(0,90){\line(1,0){40}}
\put(2,12){$x_1$} \put(2,28){$x_k$} \put(2,44){$y_1$}
\put(2,60){$y_k$} \put(2,76){$z_1$} \put(2,92){$z_k$}}
\put(110,40){\multiput(0,0)(0,40){3}{\put(0,0){\circle*{4}}
\put(40,0){\circle*{4}} \put(80,0){\circle*{4}}
\put(0,0){\line(1,0){80}} \qbezier(0,0)(40,20)(80,0)}
\put(80,0){\line(0,1){80}} \qbezier(80,0)(100,40)(80,80)
\put(74,-8){$x$} \put(36,-8){$x'$} \put(-2,-8){$x''$}
\put(74,32){$y$} \put(36,32){$y'$} \put(-2,32){$y''$}
\put(74,72){$z$} \put(36,72){$z'$} \put(-2,72){$z''$}}
\put(45,10){$G_3$ ($k\geq 8$ and $r\geq 3k+7$)}}

\end{picture}

\caption{Three classes of claw-\emph{o}-heavy non-hamiltonian
graphs}\label{FiCounterexample}
\end{center}
\end{figure}
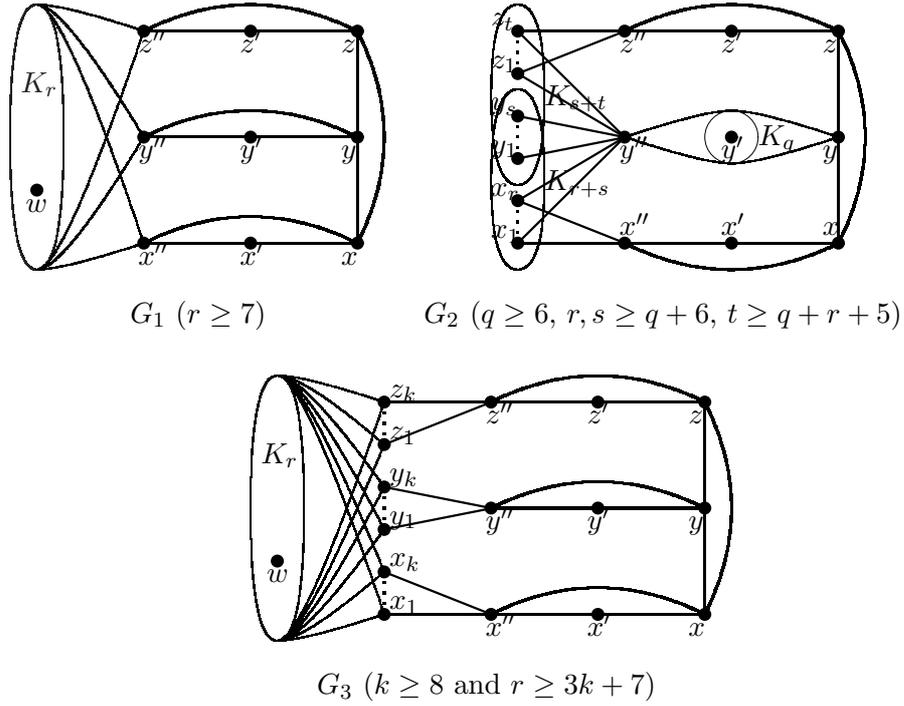

\begin{claim}\label{Cl1If}
None of $\{22,23,24,33,34,35,44,45,55\}$ is in $E(\gamma)$.
\end{claim}

\begin{proof}
Recall that $E(\gamma)\backslash E(\gamma_1)\neq\emptyset$, i.e.,
one of $\{11,12,13,22,23,33,44,45,46,55,56,66\}$ is in $E(\gamma)$.
Since $\gamma$ is symmetrical, one of $\{11,12,13,22,23,33\}$ is in
$E(\gamma)$ and one of $\{44,45,46,55,56,66\}$ is in $E(\gamma)$.

Suppose that one of $\{22,23,24,33,34,35,44,45,55\}$ is in
$E(\gamma)$. Since $\gamma$ is symmetrical, one of
$\{22,23,24,33,34,44\}$ is in $E(\gamma)$ and one of
$\{33,34,35,44,45,55\}$ is in $E(\gamma)$. Consider the graph $G_1$.
Let $P=v_1v_2\cdots v_6$ be an induced path of $G_1$, and let $ij$
be an edge in $E(\gamma)$ such that
$$ij\in\left\{\begin{array}{ll}
\{11,12,13,22,23,33\}, & \mbox{ if } P=x'xyy''wz'';\\
\{22,23,24,33,34,44\}, & \mbox{ if } P=xyy''wz''z';\\
\{33,34,35,44,45,55\}, & \mbox{ if } P=x'x''wy''yz;\\
\{44,45,46,55,56,66\}, & \mbox{ if } P=xx''wy''yzz'.
\end{array}\right.$$
Then $d(v_i)+d(v_j)\geq|V(G_1)|$. Note that $G_1$ has only the four
essentially different induced copies of $P_6$. This implies that
$G_1$ is $P_6$-$\gamma$-heavy, a contradiction.
\end{proof}

Let $\mathfrak{E}_1=\{22,23,24,33,34,35,44,45,55\}$. Then for
$k=1,2,3$,
$E(\gamma)\backslash(E(\gamma_k)\cup\mathfrak{E}_1)\neq\emptyset$.
Note that
$E(\overline{\gamma_2})\backslash\mathfrak{E}_1=\{13,46\}$. Since
$\gamma$ is symmetrical, we can see that $13,46\in E(\gamma)$.

\begin{claim}\label{Cl2If}
None of $\{11,16,66\}$ is in $E(\gamma)$.
\end{claim}

\begin{proof}

Suppose not. Since $\gamma$ is symmetrical, we can see that one of
$\{11,16\}$ is in $E(\gamma)$ and one of $\{16,66\}$ is in
$E(\gamma)$.

Consider the graph $G_3$. Let $P=v_1v_2\cdots v_6$ be an induced
path of $G_3$, and let $ij$ be an edge in $E(\gamma)$ such that
$$ij=\left\{\begin{array}{ll}
13,                 & \mbox{ if } P=wx_1x''xyy';\\
11 \mbox{ or } 16,  & \mbox{ if } P=x_1x''xyy''y_1;\\
46,                 & \mbox{ if } P=x'xyy''y_1w;\\
46,                 & \mbox{ if } P=xyy''y_1wz_1;\\
46,                 & \mbox{ if } P=x'x''x_1wy_1y'';\\
13,                 & \mbox{ if } P=x''x_1wy_1y''y';\\
13,                 & \mbox{ if } P=x_1wy_1y''yz.
\end{array}\right.$$
Then $d(v_i)+d(v_j)\geq|V(G_3)|$. Note that $G_3$ has only the seven
essentially different induced copies of $P_6$. This implies that
$G_3$ is $P_6$-$\gamma$-heavy, a contradiction.
\end{proof}

Let $\mathfrak{E}_2=\mathfrak{E}_1\cup\{11,16,66\}$. By Claims
\ref{Cl1If} and \ref{Cl2If},
$E(\gamma)\backslash(E(\gamma_3)\cup\mathfrak{E}_2)\neq\emptyset$.
Note that
$E(\overline{\gamma_3})\backslash\mathfrak{E}_2=\{12,56\}$. Since
$\gamma$ is symmetrical, we can see that $12,56\in E(\gamma)$.

Let $\gamma'$ be a graph on $\mathfrak{I}$ with edge set
$E(\gamma')=\{12,13,46,56\}$. Then $\gamma'$ is a subgraph of
$\gamma$. Similarly as in Claim \ref{Cl2If}, one can check that
$G_2$ is $P_6$-$\gamma'$-heavy, and then is $P_6$-$\gamma$-heavy, a
contradiction. This completes the proof of the `only if' part of
Theorem \ref{ThIffi}.

\end{document}